\documentclass[generic]{imsart}

\RequirePackage[OT1]{fontenc}
\RequirePackage{amsthm,amsmath}
\RequirePackage{natbib}
\RequirePackage[colorlinks,citecolor=blue,urlcolor=blue]{hyperref}

\arxiv{arXiv:1510.01188}

\startlocaldefs
\numberwithin{equation}{section}
\theoremstyle{plain}
\newtheorem{thm}{Theorem}[section]
\newtheorem{lemma}{Lemma}[section]
\newtheorem{coro}{Corollary}[section]

\newcommand{\Bc}{\mathcal{B}}

\newcommand{\Cov}{\textnormal{Cov}}

\newcommand{\der}{\textnormal{d}}

\newcommand{\refEq}[1]{Eq.~{\eqref{#1}}}

\endlocaldefs

\begin{document}

\begin{frontmatter}
\title{Analytic Posteriors for Pearson's Correlation Coefficient}
\runtitle{Analytic Correlation Posterior}

\begin{aug}
\author{\fnms{Alexander} \snm{Ly}\thanksref{t1}\ead[label=e1]{a.ly@uva.nl}},
\author{\fnms{Maarten} \snm{Marsman}\thanksref{t1}\ead[label=e2]{m.marsman@uva.nl}}
\and
\author{\fnms{Eric-Jan} \snm{Wagenmakers} \thanksref{t1} \ead[label=e3]{ej.wagenmakers@gmail.com}}
\ead[label=u1]{https://jasp-stats.org/}

\thankstext{t1}{This work was supported by the starting grant ``Bayes or Bust'' awarded by the European Research Council (Grant number: 283876). The authors thank Christian Robert, Fabian Dablander, Tom Koornwinder and an anonymous reviewer for helpful comments that improved an earlier version of this manuscript.}
\runauthor{Ly, Marsman, \& Wagenmakers}

\affiliation{University of Amsterdam}

\address{University of Amsterdam\\
Weesperplein 4\\
1018 XA  Amsterdam \\
The Netherlands\\
\phantom{E-mail:\ }\printead*{e1}}
\end{aug}

\begin{abstract}
Pearson's correlation is one of the most common measures of linear dependence. Recently, \citet{bernardo2015obayes} introduced a flexible class of priors to study this measure in a Bayesian setting. For this large class of priors we show that the (marginal) posterior for Pearson's correlation coefficient and all of the posterior moments are analytic. Our results are available in the open-source software package JASP.
\end{abstract}



\end{frontmatter}

\section{Introduction}
Pearson's product-moment correlation coefficient \( \rho \) is a measure of the linear dependency between two random variables. Its sampled version, commonly denoted by \( r \), has been well-studied by the founders of modern statistics such as Galton, Pearson, and Fisher. Based on geometrical insights \citet{fisher1915frequency, fisher1921probable} was able to derive the exact sampling distribution of \( r \), and established that this sampling distribution converges to a normal distribution as the sample size increases. Fisher's study of the correlation has lead to the discovery of variance-stabilizing transformations, sufficiency \citep{fisher1920mathematical}, and, arguably, the maximum likelihood estimator \citep{fisher1922mathematical, stigler2007epic}. %
Similar efforts were made in Bayesian statistics which focus on inferring the unknown \( \rho \) from the data that were actually observed. This type of analysis requires the statistician to (i) choose a prior on the parameters, thus, also on \( \rho \), and to (ii) calculate the posterior. Here we derive analytic posteriors for \( \rho \) given a large class of priors that include the recommendations of \citet{jeffreys1961theory}, \citet{lindley1965introduction}, \citet{bayarri1981inferencia}, and, more recently, \citet{berger2008objective} and \citet*{berger2015overall}. Jeffreys's work on the correlation coefficient can also be found in the second edition of his book \citep{jeffreys1961theory}, originally published in 1948; see \citet*{robert2009harold} for a modern re-read of Jeffreys's work. An earlier attempt at a Bayesian analysis of the correlation coefficient can be found in \citet{jeffreys1935some}. Before presenting the results, we first discuss some notations and recall the likelihood for the problem at hand. 

\section{Notation and Result}
Let \( (X_{1}, X_{2})' \) have a bivariate normal distribution with mean \( \vec{\mu}=(\mu_{1}, \mu_{2})' \) and covariance matrix 
\begin{align}
\nonumber
\Sigma = \begin{pmatrix} \sigma_{1}^{2} & \rho \sigma_{1} \sigma_{2} \\ \rho \sigma_{1} \sigma_{2} & \sigma^{2}_{2} \end{pmatrix},
\end{align}
where \( \sigma_{1}^{2}, \sigma_{2}^{2} \) are the population variances of \( X_{1} \) and \( X_{2} \), and where \( \rho \) is 
\begin{align}
\label{pearsonsRho}
\rho = {\Cov(X_{1}, X_{2}) \over \sigma_{1} \sigma_{2} }  = {E(X_{1} X_{2}) - \mu_{1} \mu_{2} \over \sigma_{1} \sigma_{2}}.
\end{align}
Pearson's correlation coefficient \( \rho \) measures the linear association between \( X_{1} \) and \( X_{2} \). In brief, the model is parametrized by the five unknowns \( \theta=(\mu_{1}, \mu_{2}, \sigma_{1}, \sigma_{2}, \rho) \). 

Bivariate normal data consisting of \( n \) pairs of observations can be sufficiently summarized as \( y=(n, \bar{x}_{1}, \bar{x}_{2}, s_{1}, s_{2}, r) \), where 
\begin{align}
\nonumber
r=  {\sum_{j=1}^{n} (x_{1j} - \bar{x}_{1}) (x_{2j} - \bar{x}_{2}) \over n s_{1} s_{2}},
\end{align}
is the sample correlation coefficient, \( \bar{x}_{i}=\tfrac{1}{n} \sum_{j=1}^{n} x_{ij} \) the sample mean and \( s_{i}^{2} = \tfrac{1}{n} \sum_{j=1}^{n} (x_{ij} - \bar{x}_{i})^{2}  \) the average sums of squares. The bivariate normal model implies that the observations \( y \) are functionally related to the parameters by the following likelihood function %
\begin{align}
\nonumber
f(y  \, | \, \theta) =  & \big ( 2 \pi \sigma_{1} \sigma_{2} \sqrt{1 - \rho^{2}} \big ) ^{-n} \exp \big ( - \tfrac{n}{2 (1-\rho^{2})}  \big [ \tfrac{(\bar{x}_{1} - \mu_{1})^{2}}{\sigma_{1}^{2}}  - 2 \rho \tfrac{(\bar{x}_{1} - \mu_{1})(\bar{x}_{2} - \mu_{2})}{\sigma_{1} \sigma_{2}} +  \tfrac{(\bar{x}_{2} - \mu_{2})^{2}}{\sigma_{2}^{2}}  \big] \big ) \\
\label{likelihoodCorM1}
& \times \exp \big ( - \tfrac{n}{2(1-\rho^{2})}  \Big [ \big ( \tfrac{s_{1}}{\sigma_{1}} \big )^{2}   - 2 \rho \big ( \tfrac{r s_{1} s_{2}}{\sigma_{1} \sigma_{2} } \big ) + \big ( \tfrac{s_{2}}{\sigma_{2}} \big )^{2} \big ] \Big ). 
\end{align}
For inference, we use the following class of priors %
\begin{align}
\label{corPrior}
\pi_{\eta}(\theta) \propto \underbrace{(1- \rho^{2})^{\alpha-1} (1+\rho^{2})^{\tfrac{\beta}{2}}}_{\pi_{\alpha, \beta}(\rho)} \underbrace{\sigma_{1}^{\gamma-1}}_{\pi_{\gamma}(\sigma_{1})} \underbrace{\sigma_{2}^{\delta-1}}_{\pi_{\delta}(\sigma_{2})},
\end{align}
where \( \eta \) denotes the hyperparameters, that is, \( \eta=(\alpha, \beta, \gamma, \delta) \). This class of priors is inspired by the one Jos\'{e} Bernardo used in his talk on reference priors for the bivariate normal distribution at the 11th International Workshop on Objective Bayes Methodology in honor of Susie Bayarri. This class of priors contains certain recommended priors as special cases. 

If we set \( \alpha=1, \beta=\gamma=\delta=0 \) in \refEq{corPrior}, we retrieve the prior that Jeffreys recommended for both estimation and testing \citep[pp. 174--179 and 289--292]{jeffreys1961theory}. This recommendation is \emph{not} the prior derived from Jeffreys's rule based on the Fisher information (e.g., \citealp{ly2016tutorial}), as discussed in \citet{berger2008objective}. With \( \alpha=1, \beta=\gamma=\delta=0 \), thus, a uniform prior on \( \rho \), Jeffreys showed that the marginal posterior for \( \rho \) is proportional to \( h_{a}(n, r \, | \, \rho) \), where %
\begin{align}
\nonumber
h_{\text{a}}(n, r  \, | \, \rho) = (1-\rho^{2})^{\tfrac{n-1}{2}}(1-\rho r)^{\tfrac{3-2n} {2}},
\end{align}
represents the \( \rho \) dependent part of the likelihood \refEq{likelihoodCorM1} with \( \theta_{0}=(\mu_{1}, \mu_{2}, \sigma_{1},\sigma_{2}) \) integrated out. %
For \( n \) large enough, the function \( h_{\text{a}} \) is a good approximation to the true reduced likelihood \( h_{\gamma, \delta} \), given below.%
\footnote{We thank an anonymous reviewer for clarifying how Jeffreys derived this approximation.} %

If we set \( \alpha=\beta=\gamma=\delta=0 \) in \refEq{corPrior}, we retrieve Lindley's reference prior for \( \rho \). \citet[pp 214--221]{lindley1965introduction} established that the posterior of \( \tanh^{-1}(\rho) \) is asymptotically normal with mean \( \tanh^{-1}(r) \) and variance \( n^{-1} \), which relates the Bayesian method of inference for \( \rho \) to that of Fisher. In Lindley's \citeyearpar[p. 216]{lindley1965introduction} derivation it is explicitly stated that the likelihood with \( \theta_{0} \) integrated out cannot be expressed in terms of elementary functions. In his analysis, Lindley approximates the true reduced likelihood \( h_{\gamma, \delta} \) with the same \( h_{\text{a}} \) that Jeffreys used before. \cite{bayarri1981inferencia} furthermore showed that with the choice \( \gamma=\delta=0 \) the marginalization paradox \citep{dawid1973marginalization} is avoided.

In their overview, \citet{berger2008objective} showed that for certain \( a, b \) with \( \alpha=b/2-1 \), \( \beta=0 \), \( \gamma=a-2 \) and \( \delta=b-1 \) the priors in \refEq{corPrior} correspond to a subclass of the generalized Wishart distribution. Furthermore, a right-Haar prior (e.g., \citealp{sun2007objective}) is retrieved when we set \( \alpha=\beta=0 \), \( \gamma=-1, \delta=1 \) in \refEq{corPrior}. This right-Haar prior then has a posterior that can be constructed through simulations. That is, by simulating from a standard normal distribution and two chi-squared distributions \cite[Table~{1}]{berger2008objective}. This constructive posterior also corresponds to the fiducial distribution for \( \rho \) (e.g, \citealp{fraser1961fiducial}, \citealp{hannig2006fiducial}). Another interesting case is given by \( \alpha=0, \beta=1, \gamma=\delta=0 \), which corresponds to the one-at-a-time reference prior for \( \sigma_{1} \) and \( \sigma_{2} \), see also \citet[p. 187]{jeffreys1961theory}. 

The analytic posteriors given below follow directly from exact knowledge of the reduced likelihood \( h_{\gamma, \delta}(n, r \, | \, \rho) \) rather than its approximation used in previous work. As we have not encountered this proof in earlier work, full details are given below. 

\begin{thm}[The Reduced Likelihood $h_{\gamma, \delta}(n,r \, | \, \rho)$]
If \( |r| <1 \), \( n > \gamma+1 \) and \( n > \delta +1 \), then the likelihood \( f(y \, | \, \theta) \) times the prior \refEq{corPrior} with \( \theta_{0}=(\mu_{1}, \mu_{2}, \sigma_{1}, \sigma_{2}) \) integrated out is a function \( f_{\gamma, \delta} \) that factors into %
\begin{align}
f_{\gamma, \delta}(y \, | \, \rho)=p_{\gamma, \delta}(y_{0}) h_{\gamma, \delta}(n, r  \, | \, \rho).
\end{align}
The first factor is the marginal likelihood with \( \rho \) fixed at zero, which does not depend on \( r \) nor on \( \rho \), that is,
\begin{align}
\nonumber
p_{\gamma, \delta}(y_{0}) = & \int \int \int \int f(y  \, | \, \theta_{0}, \rho=0) \pi_{\gamma}(\sigma_{1})  \pi_{\delta}(\sigma_{2})  \der \mu_{1} \der \mu_{2} \der  \sigma_{1} \der  \sigma_{2} \\
\label{marginalLikelihoodM0}
= & 2^{\tfrac{-\gamma-\delta-4}{2}} {\pi^{1-n} \over n} (n s_{1}^{2})^{\tfrac{1+\gamma-n}{2}} (n s_{2}^{2})^{\tfrac{1+\delta-n}{2}} \Gamma \big ( \tfrac{n-\gamma-1}{2} \big ) \Gamma \big ( \tfrac{n-\delta-1}{2} \big ),
\end{align}
where \( y_{0}=(n, \bar{x}_{1}, \bar{x}_{2}, s_{1}, s_{2}) \). We refer to the second factor as the reduced likelihood, a function of \( \rho \) which is given by a sum of an even and an odd function, that is, \( h_{\gamma, \delta}=A_{\gamma, \delta}+B_{\gamma, \delta} \) where
\begin{align}
 A_{\gamma, \delta}(n, r \, | \, \rho) = & (1-\rho^{2})^{\tfrac{n-\gamma-\delta-1}{2}} {_2 F}_{1} \big ( \tfrac{n-\gamma-1}{2}, \tfrac{n-\delta-1}{2} \, ; \, \tfrac{1}{2} \, ; \, r^{2} \rho^{2} \big ), \\
\label{bFunction}
B_{\gamma, \delta}(n, r \, | \, \rho) = & 2 r \rho (1-\rho^{2})^{\tfrac{n-\gamma-\delta-1}{2}} W_{\gamma, \delta}(n) {_2 F}_{1} \big ( \tfrac{n-\gamma}{2}, \tfrac{n-\delta}{2} \, ; \, \tfrac{3}{2} \, ; \, r^{2} \rho^{2} \big ),
\end{align}
where \( W_{\gamma,\delta}(n)= \Big [ \Gamma \big ( \tfrac{n-\gamma}{2} \big ) \Gamma \big ( \tfrac{n - \delta}{2} \big ) \Big ] \big / \Big [ \Gamma \big ( \tfrac{n- \gamma - 1}{2} \big ) \Gamma \big ( \tfrac{n-\delta - 1}{2} \big ) \Big ] \) and where \( _{2}F_{1} \) denotes Gauss' hypergeometric function. \( \hfill \diamond \)
\end{thm}

\begin{proof}
To derive \( f_{\gamma, \delta}(y \, | \, \rho) \) we have to perform three integrals: (i) with respect to \( \pi(\mu_{1}, \mu_{2}) \propto 1 \), (ii) \( \pi_{\gamma}(\sigma_{1}) \propto \sigma_{1}^{\gamma-1} \), and (iii) \( \pi_{\delta}(\sigma_{2}) \propto \sigma_{2}^{\delta-1} \). 

\noindent (i) The integral with respect to \( \pi(\mu_{1}, \mu_{2}) \propto 1 \) yields
\begin{align}
\label{likelihoodMeansIntegratedOut}
f (y  \, | \, \sigma_{1}, \sigma_{2}, \rho) = \tfrac{\big ( 2 \pi \sqrt{1 - \rho^{2}} \sigma_{1} \sigma_{2} \big ) ^{1-n}}{n} \exp \big ( \tfrac{-n}{2(1-\rho^{2})}  \big [  \tfrac{s_{1}^{2}}{\sigma_{1}^{2}}    - 2 \rho  \tfrac{r s_{1} s_{2}}{\sigma_{1} \sigma_{2} } + \tfrac{s_{2}^{2}}{\sigma_{2}}  \big ] \big ),
\end{align}
where we abbreviated \( f(y  \, | \, \sigma_{1}, \sigma_{2}, \rho) = \int_{-\infty}^{\infty} \int_{-\infty}^{\infty} f(y  \, | \, \theta_{0}, \rho) \text{d} \mu_{1} \text{d} \mu_{2} \). The factor \( p_{\gamma, \delta}(y_{0}) \) follows directly by setting \( \rho \) to zero in \refEq{likelihoodMeansIntegratedOut} and two independent gamma integrals with respect to \( \sigma_{1} \) and \( \sigma_{2} \) resulting in \refEq{marginalLikelihoodM0}. %
These gamma integrals cannot be used when \( \rho  \) is not zero. For \( f_{\gamma, \delta}(y \, | \, \rho) \) which is a function of \( \rho \), we use results from special functions theory. 

\noindent (ii) For the second integral, we collect only that part of \refEq{likelihoodMeansIntegratedOut} that involves \( \sigma_{1} \) into a function \( g \), that is,
\begin{align}
\nonumber
\int_{0}^{\infty} g(y \, | \, \sigma_{1}) \pi_{\gamma}(\sigma_{1}) \text{d} \sigma_{1} = \int_{0}^{\infty} \sigma_{1}^{\gamma - n} \exp \big ( - \tfrac{n s_{1}^{2}}{2(1-\rho^{2})}   \tfrac{1}{\sigma_{1}^{2}}    + \tfrac{n s_{1} s_{2} }{\sigma_{2} (1-\rho^{2})} r \rho \tfrac{1}{\sigma_{1}} \big ) \text{d} \sigma_{1}. 
\end{align}
The assumption \( n > \gamma +1 \) and the substitution \( u=\sigma_{1}^{-1} \) allows us to solve this integral using Lemma~{\ref{batemanLemma}}, which we distilled from the Bateman manuscript project \citep{bateman1954integral1}, with \( a=\tfrac{ n s_{1}^{2}}{2(1- \rho^{2})}, b=-\tfrac{n s_{1} s_{2}}{(1- \rho^{2}) \sigma_{2}} r \rho \) and \( c=n-\gamma-1 \). This yields 
\begin{align}
\int_{0}^{\infty} g(y \, | \, \sigma_{1}) \pi_{\gamma}(\sigma_{1}) \text{d} \sigma_{1} = 2^{\tfrac{n-\gamma-3}{2}} \big ( \tfrac{1- \rho^{2}}{n s_{1}^{2}} \big )^{\tfrac{n - \gamma - 1}{2}} \Big [ \mathring{A}_{\gamma} + \mathring{B}_{\gamma} \Big ] ,
\end{align}
where 
\begin{align}
\mathring{A}_{\gamma} = & \Gamma \big ( \tfrac{n-\gamma-1}{2} \big ) {_1 F}_{1} \big ( \tfrac{n-\gamma-1}{2} \, ; \, \tfrac{1}{2} \, ; \,  \tfrac{ n s_{2}^{2} (r \rho)^{2}}{2(1- \rho^{2}) }  \tfrac{1}{\sigma_{2}^{2}} \big ), \\
\mathring{B}_{\gamma} = & \sqrt{\tfrac{2 n s_{2}^{2} (r \rho)^{2}}{( 1- \rho^{2}) }}  \sigma_{2}^{-1} \Gamma \big ( \tfrac{n-\gamma}{2} \big ) {_1 F}_{1} \big ( \tfrac{n-\gamma}{2} \, ; \, \tfrac{3}{2} \, ; \,  \tfrac{ n s_{2}^{2} (r \rho)^{2}}{2(1- \rho^{2}) } \tfrac{1}{\sigma_{2}^{2}} \big ),
\end{align}
and where \( {_1 F}_{1} \) denotes the confluent hypergeometric function. The functions \( \mathring{A}_{\gamma} \) and \( \mathring{B}_{\gamma} \) are the even and odd solution of Weber's differential equation in the variable \( z=(r \rho)^{2} \tfrac{ n s_{2}^{2}}{2(1- \rho^{2}) \sigma_{2}^{2}} \) respectively. 

\noindent (iii) With \( f_{\gamma}(y \, | \, \sigma_{2}, \rho)=\int_{0}^{\infty} f(y  \, | \, \sigma_{1}, \sigma_{2}, \rho) \pi_{\gamma}(\sigma_{1}) \text{d} \sigma_{1} \), we see that \( f_{\gamma, \delta}(y \, | \, \rho) \) follows from integrating  \( \sigma_{2} \) out of the following expression%
\begin{align}
\nonumber
f_{\gamma}(y \, | \, \sigma_{2}, \rho) \pi_{\delta}( \sigma_{2})  =2^{\tfrac{-n-\gamma-1}{2}} \tfrac{\pi^{1-n}}{n} (n s_{1}^{2})^{\tfrac{1+\gamma-n}{2}} (1-\rho^{2})^{\tfrac{-\gamma}{2}} \big [ \breve{A}_{\gamma}(y  \, | \, \sigma_{2}, \rho) + \breve{B}_{\gamma}(y  \, | \, \sigma_{2}, \rho) \big ],
\end{align}
where %
\begin{align}
\label{likelihoodMeansAndSigmaIntegratedOut}
\breve{A}_{\gamma} = & \Gamma \big ( \tfrac{n-\gamma - 1}{2} \big ) \overbrace{\sigma_{2}^{\delta-n} e^{  - \tfrac{ n s_{2}^{2}}{2(1- \rho^{2})} \tfrac{1}{\sigma_{2}^{2}} } \, {_{1} F}_{1} \big ( \tfrac{n-\gamma -1}{2} \, ; \, \tfrac{1}{2} \, ; \,  (r \rho)^{2} \tfrac{ n s_{2}^{2}}{2(1- \rho^{2})}  \tfrac{1}{\sigma_{2}^{2}} \big )}^{k(n, r  \, | \, \rho, \sigma_{2})},  \\
\nonumber
\breve{B}_{\gamma} = & (\tfrac{2 n s_{2}^{2}}{1-\rho^{2}})^{\tfrac{1}{2}} r \rho \Gamma \big ( \tfrac{n-\gamma}{2} \big ) \underbrace{\sigma_{2}^{\delta-n -1} e^{ - \tfrac{ n s_{2}^{2}}{2(1- \rho^{2})} \tfrac{1}{\sigma_{2}^{2}} } \, {_{1} F}_{1} \big ( \tfrac{n-\gamma}{2} \, ; \, \tfrac{3}{2} \, ; \, (r \rho)^{2} \tfrac{ n s_{2}^{2}}{2(1- \rho^{2}) }  \tfrac{1}{\sigma_{2}^{2}} \big )}_{l(n, r  \, | \, \rho, \sigma_{2})}.
\end{align}
Hence, the last integral with respect to \( \sigma_{2} \) only involves the functions \( k \) and \( l \) in \refEq{likelihoodMeansAndSigmaIntegratedOut}. The assumption \( n > \delta +1 \) and the substitution \( t=\frac{n s_{2}^{2}}{2 (1- \rho^{2})} \sigma_{2}^{-2} \), thus, \( \text{d} \sigma_{2} = - \tfrac{1}{2} \sqrt{ \frac{n s_{2}^{2}}{2 ( 1 - \rho^{2})}} t^{-\tfrac{3}{2}} \der  t \) allows us to solve this integral using Eq.~{7.621.4} from \citet[p. 822]{gradshteyn2007table} with \( s=1 \), \( \tilde{k}=(r \rho)^{2} \). This yields %
\begin{align}
\nonumber
\int_{0}^{\infty} k(n, r  \, | \, \rho, \sigma_{2}) \text{d} \sigma_{2} = & 2^{\tfrac{n-\delta-3}{2}} \big ( \tfrac{1-\rho^2}{n s_{2}^2} \big )^{\tfrac{n-\delta-1}{2}} \Gamma \big ( \tfrac{n-\delta  -1}{2} \big ) {_2 F}_{1} \big ( \tfrac{n-\gamma -1}{2}, \tfrac{n-\delta - 1}{2} \, ; \, \tfrac{1}{2} \, ; \, r^{2} \rho^{2} \big ), \\
\nonumber
\int_{0}^{\infty} l(n, r  \, | \, \rho, \sigma_{2}) \text{d} \sigma_{2} = & 2^{\tfrac{n-\delta -2}{2}} \big ( \tfrac{1- \rho^2}{n s_{2}^2} \big )^{\tfrac{n-\delta}{2}} \Gamma \big ( \tfrac{n-\delta}{2} \big ) {_2 F}_{1} \big ( \tfrac{n-\gamma}{2}, \tfrac{n-\delta}{2} \, ; \, \tfrac{3}{2} \, ; \, r^{2} \rho^{2} \big ).
\end{align}
After we combine the results we see that \( f_{\gamma, \delta}(y  \, | \, \rho)=\tilde{A}_{\gamma, \delta}(y \, | \, \rho) + \tilde{B}_{\gamma, \delta}(y  \, | \, \rho) \), where 
\begin{align}
\nonumber
\frac{\tilde{A}_{\gamma, \delta}(y \, | \, \rho)}{p_{\gamma, \delta}(y_{0})}= & (1-\rho^{2})^{\tfrac{n-\gamma-\delta-1}{2}} {_2 F}_{1} \big ( \tfrac{n-\gamma-1}{2}, \tfrac{n-\delta-1}{2} \, ; \, \tfrac{1}{2} \, ; \, r^{2} \rho^{2} \big ), \\
\nonumber
\frac{\tilde{B}_{\gamma, \delta}(y \, | \, \rho)}{ p_{\gamma, \delta}(y_{0})} = & 2 r \rho (1-\rho^{2})^{\tfrac{n-\gamma-\delta-1}{2}} W_{\gamma, \delta}(n) {_2 F}_{1} \big ( \tfrac{n-\gamma}{2}, \tfrac{n-\delta}{2} \, ; \, \tfrac{3}{2} \, ; \, r^{2} \rho^{2} \big ).
\end{align}
Hence, \( f_{\gamma, \delta}(y \, | \, \rho) \) is of the asserted form. Note that \( A_{\gamma, \delta} =\frac{\tilde{A}_{\gamma, \delta}(y \, | \, \rho)}{ p_{\gamma, \delta}(y_{0})}  \) is even, while \( \frac{\tilde{B}_{\gamma, \delta}(y \, | \, \rho)}{ p_{\gamma, \delta}(y_{0})} \) is an odd function of \( \rho \).
\end{proof}

This main theorem confirms Lindley's insights; \( h_{\gamma, \delta}(n,r \, | \, \rho) \) is indeed not expressible in terms of elementary functions and the prior on \( \rho \) is updated by the data only through its sampled version \( r \) and the sample size \( n \). As a result, the marginal likelihood for data \( y \) then factors into \( p_{\eta}(y) = p_{\gamma, \delta}(y_{0}) p^{\gamma, \delta}_{\alpha, \beta}(n, r) \), where \( p^{\gamma, \delta}_{\alpha, \beta}(n, r) = \int h_{\gamma, \delta}(n, r \, | \, \rho) \pi_{\alpha, \beta}(\rho) \der \rho \) is the normalizing constant of the marginal posterior of \( \rho \). More importantly, the fact that the reduced likelihood is the sum of an even and an odd function allows us to fully characterize the posterior distribution of \( \rho \) for the priors \refEq{corPrior} in terms of its moments. These moments are easily computed, as the prior \( \pi_{\alpha, \beta}(\rho) \) itself is symmetric around zero. Furthermore, the prior \( \pi_{\alpha, \beta}(\rho) \) can be normalized as 
\begin{align}
\label{priorRho}
 \pi_{\alpha, \beta}(\rho) =  {(1-\rho^{2})^{\alpha-1} (1+\rho^{2})^{\tfrac{\beta}{2}} \over \Bc \big (\tfrac{1}{2}, \alpha \big ) {_{2} F}_{1}   \big ( - \tfrac{\beta}{2}, \tfrac{1}{2}  \, ; \,  \tfrac{1}{2} + \alpha   \, ; \, -1  \big )}, 
\end{align}
where \( \Bc(u, v) = \frac{\Gamma ( u) \Gamma (v)}{ \Gamma ( u+ v)} \) denotes the beta function. The case with \( \beta=0 \) is also known as the (symmetric) stretched beta distribution on \( (-1,1) \) and leads to Lindley's reference prior when we ignore the normalization constant, i.e., \( \Bc(\tfrac{1}{2}, \alpha) \), and, subsequently, let \( \alpha \rightarrow 0 \). 

\begin{coro}[Characterization of the Marginal Posteriors of $\rho$]
If \( n > \gamma +\delta - 2\alpha+1 \), then the main theorem implies that whenever the marginal likelihood with all the parameters integrated out factors as \( p_{\eta}(y)=p_{\gamma, \delta}(y_{0}) p^{\gamma, \delta}_{\alpha, \beta}(n,r) \), where 
\begin{align}
p^{\gamma, \delta}_{\alpha, \beta}(n,r) = \int_{-1}^{1} h_{\gamma, \delta}(n, r  \, | \, \rho)  \pi_{\alpha, \beta}(\rho) \text{d} \rho =  \int_{-1}^{1} A_{\gamma, \delta} (n, r  \, | \, \rho) \pi_{\alpha, \beta}(\rho) \text{d} \rho ,
\end{align}
defines the normalizing constant of the marginal posterior for \( \rho \). Observe that the integral involving \( B_{\gamma, \delta} \) is zero, because \( B_{\gamma, \delta} \) is odd on \( (-1,1) \). More generally, the \( k \)th posterior moment of \( \rho \) is 
\begin{align}
E(\rho^{k}  \, | \, n, r) = 
\begin{cases} \frac{1}{ p^{\gamma, \delta}_{\alpha, \beta}(n,r) } \int \limits_{-1}^{1} \rho^{k} A_{\gamma, \delta}(n, r  \, | \, \rho) \pi_{\alpha, \beta}(\rho) \text{d} \rho & \mbox{ if } k \text{ is even}, \\
\frac{1}{ p^{\gamma, \delta}_{\alpha, \beta}(n,r) } \int \limits_{-1}^{1}  \rho^{k} B_{\gamma, \delta}(n, r  \, | \, \rho) \pi_{\alpha, \beta}(\rho) \text{d} \rho & \mbox{ if }  k \text{ is odd}.
\end{cases}
\end{align}
These posterior moments define the series 
\begin{align}
\label{posteriorMoments}
 E(\rho^{k}  \, | \, n, r) = 
 \begin{cases}
{1 \over C_{\alpha, \beta}} \sum \limits_{m=0}^{\infty} \tfrac{ \big (\tfrac{n- \gamma -1}{2} \big )_{m}  \big (\tfrac{n- \delta -1}{2} \big )_{m}}{ \big (\tfrac{1}{2} \big )_{m} m!} a_{k, m} r^{2 m} & \text{ if } k \text{ is even}, \\
{2 W_{\gamma, \delta}(n) \over C_{\alpha, \beta}} \sum \limits_{m=0}^{\infty} \tfrac{ \big (\tfrac{n- \gamma }{2} \big )_{m}  \big (\tfrac{n- \delta}{2} \big )_{m}}{ \big (\tfrac{3}{2} \big )_{m} m! } b_{k, m} r^{2 m+1} & \text{ if } k \text{ is odd}, \\
 \end{cases}
\end{align}
where \( C_{\alpha, \beta}=\Bc  (\tfrac{1}{2}, \alpha   ) {_{2} F}_{1}  ( \tfrac{-\beta}{2}, \tfrac{1}{2}  \, ; \,  \alpha + \tfrac{1}{2}  \, ; \, -1  ) \) is the normalization constant of the prior \refEq{priorRho}, \( W_{\gamma, \delta}(n) \) is the ratios of gamma functions as defined under \refEq{bFunction} and \( (x)_{m} = {\Gamma(x+m) \over \Gamma (x)}= x (x+1) (x+2) \ldots (x+m-1)  \) refers to the Pochhammer symbol for rising factorials. The terms \( a_{k, m} \) and \( b_{k,m} \) are 
\begin{align}
\nonumber
a_{k, m} & =  { \Bc \big ( \tfrac{1}{2} + \tfrac{k+2 m}{2} ,  \alpha+ \tfrac{ n - \gamma -\delta -1}{2} \big ) {_{2} F}_{1}  \big ( \tfrac{-\beta}{2}, \tfrac{k+2 m +1}{2}  \, ; \,   \tfrac{k+2 m + 2 \alpha + n - \gamma - \delta}{2}  \, ; \, -1 \big  )} , \\
\nonumber
b_{k, m} & = {\Bc \big ( \tfrac{1}{2} + \tfrac{k + 2 m + 1}{2}, \alpha + \tfrac{ n - \gamma - \delta -1}{2} \big ) {_{2} F}_{1}  \big (  \tfrac{-\beta}{2}, \tfrac{k+2 m +2}{2}  \, ; \,   \tfrac{k+2 m + 2 \alpha + n - \gamma - \delta+1}{2}  \, ; \, -1 \big  ) } .
\end{align}
The series defined in \refEq{posteriorMoments} are hypergeometric when \( \beta \) is a non-negative integer. \( \hfill \diamond \)
\end{coro}

\begin{proof}
The series \(  E(\rho^{k}  \, | \, n, r) \) result from term-wise integration of the hypergeometric functions in \( A_{\gamma, \delta} \) and \( B_{\gamma, \delta} \). The assumption \( n > \gamma +\delta - 2\alpha+1 \) and the substitution \( x=\rho^{2} \) allows us to solve these integrals using Eq.~{(3.197.8)} in \citet[p. 317]{gradshteyn2007table} with their \( \tilde{\alpha}=1 \), \( u=1 \), \( \lambda=\tfrac{\beta}{2} \), \( \mu=\alpha + \tfrac{n - \gamma - \delta -1}{2} \) and \( \nu=\tfrac{1}{2} + \tfrac{k+2m}{2} \) when \( k \) is even, while we use \( \nu= \tfrac{1}{2} + \tfrac{k+2m+1}{2} \) when \( k \) is odd. A direct application of the ratio test shows that the series converge when \( |r| < 1 \). 
\end{proof}

\section{Analytic Posteriors for the Case $\beta=0$}
For most of the priors discussed above we have \( \beta=0 \), which leads to the following simplification of the posterior. 

\begin{coro}[Characterization of the Marginal Posteriors of $\rho$, when $\beta=0$]
If \( n > \gamma + \delta - 2 \alpha + 1 \) and \( |r| < 1 \), then the marginal posterior for \( \rho \) is 
\begin{align}
\label{posteriorRho}
\pi(\rho  \, | \, n, r) & = \frac{ (1-\rho^{2})^{\tfrac{2 \alpha+n-\gamma-\delta-3}{2}}   }{  p^{\gamma, \delta}_{\alpha}(n, r) \Bc \big (\tfrac{1}{2}, \alpha \big )} \\ 
\nonumber
\times & \Big [ {_{2} F}_{1}  ( \tfrac{n-\gamma-1}{2}, \tfrac{n-\delta - 1}{2}  \, ; \,  \tfrac{1}{2}  \, ; \, r^{2} \rho^{2} ) + 2 r \rho W_{\gamma, \delta}(n)  {_{2} F}_{1}  ( \tfrac{n-\gamma}{2}, \tfrac{n-\delta}{2}  \, ; \,  \tfrac{3}{2}  \, ; \, r^{2} \rho^{2} ) \Big ],
\end{align}
where \( p^{\gamma, \delta}_{\alpha}(n, r) \) refers to the normalizing constant of the (marginal) posterior of \( \rho \), which is given by
\begin{align}
\nonumber
p^{\gamma, \delta}_{\alpha}(n, r) = \Bc \big (\tfrac{1}{2}, \alpha + \tfrac{n - \gamma -\delta -1}{2} \big  ) {_{2} F}_{1}  \big ( \tfrac{n-\gamma -1}{2}, \tfrac{n-\delta -1}{2}  \, ; \,  \alpha + \tfrac{n-\gamma - \delta}{2}  \, ; \, r^{2} \big )  / \Bc \big (\tfrac{1}{2}, \alpha \big ).
\end{align}
More generally, when \( \beta=0 \), the \( k \)th posterior moment is 
\begin{align}
\nonumber
\frac{\Bc \big (  \tfrac{1}{2} + \tfrac{k}{2}, \alpha + \tfrac{n - \gamma -\delta -1}{2} \big ) {_{3} F}_{2} \big  ( \tfrac{k+1}{2}, \tfrac{n-\gamma -1}{2}, \tfrac{n-\delta -1}{2}  \, ; \,  \tfrac{1}{2}, \tfrac{k+2 \alpha + n-\gamma - \delta }{2}  \, ; \, r^{2} \big ) }{\Bc \big ( \tfrac{1}{2}, \alpha  + \tfrac{n - \gamma -\delta -1}{2} \big ) {_{2} F}_{1} \big  ( \tfrac{n-\gamma -1}{2}, \tfrac{n-\delta -1}{2}  \, ; \,  \tfrac{2 \alpha + n-\gamma - \delta}{2}  \, ; \, r^{2} \big )},
\end{align}
when \( k \) is even, and 
\begin{align}
\nonumber
2 r W_{\gamma, \delta}(n)  \frac{ \Bc \big ( \tfrac{1}{2} + \tfrac{k+1}{2}, \alpha + \tfrac{n - \gamma -\delta -1}{2} \big ) {_{3} F}_{2}  \big ( \tfrac{k+2}{2}, \tfrac{n-\gamma }{2}, \tfrac{n-\delta }{2}  \, ; \,  \tfrac{3}{2}, \tfrac{k+2 \alpha + n-\gamma - \delta+1}{2}  \, ; \, r^{2} \big )} { \Bc \big (\tfrac{1}{2}, \alpha + \tfrac{ n - \gamma -\delta -1}{2} \big ) {_{2} F}_{1}  \big ( \tfrac{n-\gamma -1}{2}, \tfrac{n-\delta -1}{2}  \, ; \,   \tfrac{2 \alpha + n-\gamma - \delta}{2}  \, ; \, r^{2} \big )},
\end{align}
when \( k \) is odd. \( \hfill \diamond \)
\end{coro}

\begin{proof}
The assumption \( n > \gamma + \delta - 2 \alpha +1 \) and the substitution \( x=\rho^{2} \) allows us to use Eq.~{(7.513.12)} in \citet[p. 814]{gradshteyn2007table} with \( \mu= \alpha + \tfrac{ n - \gamma - \delta -1}{2} \) and  \( \nu= \tfrac{1}{2} + \tfrac{k}{2} \) when \( k \) is even, while we use \( \nu=\tfrac{1}{2} + \tfrac{k+1}{2} \) when \( k \) is odd. The normalizing constant of the posterior \( p_{\alpha}^{\gamma,\delta}(n, r) \) is a special case with \( k=0 \).
\end{proof}

The marginal posterior for \( \rho \) updated from the generalized Wishart prior, the right-Haar prior and Jeffreys's recommendation then follow from a direct substitution of the values for \( \alpha,\gamma \) and \( \delta \) as discussed under \refEq{corPrior}. Lindley's reference posterior for \( \rho \) is given by %
\begin{align}
\nonumber
\frac{   \, {_{2} F_{1}}  \big ( \tfrac{n-1}{2}, \tfrac{n-1}{2} \, ; \, \tfrac{1}{2} \, ; \, r^{2} \rho^{2} \big ) +  2 r \rho W_{0,0}(n)   \, {_{2} F_{1}}  \big ( \tfrac{n}{2}, \tfrac{n}{2} \, ; \, \tfrac{3}{2} \, ; \, r^{2} \rho^{2} \big ) }{  \Bc \big ( \tfrac{1}{2}, \tfrac{n-1}{2} \big ) \, {_{2} F_{1}}  \big ( \tfrac{n-1}{2}, \tfrac{n-1}{2} \, ; \, \tfrac{n}{2} \, ; \, r^{2} \big )} (1- \rho^2)^{\tfrac{n-3}{2}},
\end{align}
which follows from \refEq{posteriorRho} by setting \( \gamma=\delta=0 \) and, subsequently, letting \( \alpha \rightarrow 0 \). 

Lastly, for those who wish to sample from the posterior distribution, we suggest the use of an independence-chain Metropolis algorithm (IMH; \citealp{tierney1994markov}) using Lindley's normal approximation of the posterior of \( \tanh^{-1}(\rho) \) as the proposal. %
This method could be used when Pearson's correlation is embedded within a hierarchical model, as the posterior for \( \rho \) will then be a full conditional distribution within a Gibbs sampler. %
For \( \alpha = 1 \), \( \beta = \gamma =\delta = 0 \), \( n=10 \) observations and \( r = 0.6 \), the 	acceptance rate of the IMH algorithm was already well above 75\%, suggesting a fast convergence of the Markov chain. For \( n \) larger, the acceptance rate further increases. The R code for the independence-chain Metropolis algorithm can be found on the first author's home page. In addition, this analysis is also implemented in the open-source software package JASP (\printead{u1}).

\appendix

\section{A Lemma distilled from the Bateman Project}
\label{appBateman}
\begin{lemma}
\label{batemanLemma}
For \( a, c > 0  \) the following equality holds
\begin{align}
\int_{0}^{\infty} u^{c-1} \exp \big ( -a u^2 - b u \big )  \text{d} u =  2^{-1} a^{-\tfrac{c}{2}} \Big [ \mathring{A}(a, b, c) +\mathring{B}(a, b, c) \Big ],
\end{align}
that is, the integral is solved by the functions 
\begin{align}
\mathring{A}(a,b,c) = &  \Gamma \big ( \tfrac{c}{2} \big ) {_1 F}_{1} \big ( \tfrac{c}{2} \, ; \, \tfrac{1}{2} \, ; \, \tfrac{b^2}{4 a} \big ), \\
\nonumber
\mathring{B}(a,b,c) = & - \tfrac{b}{\sqrt{a}} \Gamma \big ( \tfrac{c+1}{2} \big ) {_1 F}_{1} \big ( \tfrac{c+1}{2} \, ; \, \tfrac{3}{2} \, ; \, \tfrac{b^2}{4 a} \big ),
\end{align}
which define the even and odd solutions to Weber's differential equation in the variable \( z=\tfrac{b}{\sqrt{2 a}} \) respectively. \( \hfill \diamond \)
\end{lemma}

\begin{proof}
By \citet[p 313, Eq.~{13}]{bateman1954integral1} we note that, 
\begin{align}
\label{lemma1a}
\int_{0}^{\infty} u^{c-1} \exp \big ( -a u^2 - b u \big )  \text{d} v =  (2 a)^{\tfrac{-c}{2}} \Gamma (c) \exp \big ( \tfrac{b^{2}}{8 a} \big ) D_{-c} \big ( \tfrac{b}{\sqrt{2 a}} \big ) ,
\end{align}
where \( D_{\lambda} (z) \) is Whittaker's \citeyearpar{whittaker1902functions} parabolic cylinder function \citep{abramowitz1972handbook}. By virtue of Eq.~{4} on p. 117 of \citet{bateman1953higher2}, we can decompose \( D_{\lambda} (z) \) into a sum of an even and odd function. Replacing this decomposition for \( D_{\lambda}(z) \) in \refEq{lemma1a} and an application of the duplication formula of the gamma function yields the statement. 
\end{proof}



\bibliographystyle{imsart-nameyear}
\bibliography{sunBergerCorrie}
\end{document}